\documentclass[11pt,final]{article}
\usepackage{a4}
\usepackage{amsmath}%
\usepackage{amstext}%
\usepackage{amssymb}%
\usepackage{showkeys}%
\usepackage{epsfig}%
\usepackage{cite}
\usepackage{tikz}
\usepackage{subfig}
\usepackage{caption}
\usepackage{multirow}
\usepackage{booktabs}
\usepackage{floatrow}
\usepackage{comment}
\usepackage{listings}

\usepackage[margin=0.5in]{geometry}

\newtheorem{theorem}{Theorem}

\newtheorem{axiom}{Axiom}

\newtheorem{corollary}[theorem]{Corollary}

\newtheorem{definition}[axiom]{Definition}

\newtheorem{lemma}[theorem]{Lemma}

\newtheorem{rem}[theorem]{Remark}
\newenvironment{remark}{\rem\rm}{\endrem}

\newcounter{unnumber}

\newtheorem{prf}[unnumber]{Proof.}
\newenvironment{proof}{\prf\rm}{\hfill{$\blacksquare$}\endprf}
\newcommand{\R}{\mathbb{R}}%
\newcommand{\N}{\mathbb{N}}%
\newcommand{\e}{\varepsilon}%
\newcommand{\ol}{\overline}%
\newcommand{\ul}{\underline}%

\newcommand{\n}{{\nabla}}

\newcommand{\ds}{\displaystyle}
\newcommand{\To}{\longrightarrow}
\def\a{\alpha}

\def\b{\beta}
\def\e{\epsilon}
\def\d{\delta}
\def\t{\theta}
\def\g{\gamma}

\def\l{\lambda}
\def\<{\langle}
\def\>{\rangle}
\DeclareMathOperator*\prox{prox}%

\DeclareMathOperator*\crit{crit}
\DeclareMathOperator*\dist{dist}

\author{ Szil\'{a}rd Csaba L\'{a}szl\'{o} \thanks{Technical University of Cluj-Napoca, Department of Mathematics,
 Str. Memorandumului nr. 28, 400114 Cluj-Napoca, Romania, e-mail: laszlosziszi@yahoo.com This work was supported by a grant of Ministry of Research and Innovation, CNCS - UEFISCDI, project number PN-III-P1-1.1-TE-2016-0266, and  by a grant of Ministry of Research and Innovation,
 CNCS - UEFISCDI, project number PN-III-P4-ID-PCE-2016-0190, within PNCDI III.}}
\title{Convergence rates for an inertial algorithm of gradient type associated to a smooth nonconvex minimization}

\begin{document}
\maketitle

\noindent \textbf{Abstract.}  We investigate an inertial algorithm of gradient type  in connection with the minimization of a nonconvex differentiable function. The algorithm is formulated in the spirit of Nesterov's accelerated convex gradient method. We show that the generated sequences converge to a critical point of the objective function, if a regularization of the objective function satisfies the Kurdyka-{\L}ojasiewicz property. Further, we provide convergence rates for the generated sequences and the objective function values formulated in terms of the {\L}ojasiewicz exponent. \vspace{1ex}

\vspace{1ex}


\section{Introduction}

Let $g:\R^m\To \R$ be a (not necessarily convex) Fr\'{e}chet differentiable function with $L_g$-Lipschitz continuous gradient, i.e. there exists $L_g\ge 0$ such that $\|\n g(x)-\n g(y)\|\le L_g\|x-y\|$ for all $x,y\in \R^m.$ We deal with the optimization problem
\begin{equation}\label{opt-pb} (P) \ \inf_{x\in\R^m}g(x). \end{equation}

We associate to \eqref{opt-pb} the following inertial algorithm of gradient type.
Consider the starting points $x_0=y_0\in\R^m,$ and for all $n\in\N$
\begin{equation}\label{generaldiscrete}\left\{\begin{array}{lll}
\ds x_{n+1}=y_n-s\n g(y_n),\\
\\
\ds y_n=x_n+\frac{\b n}{n+\a}(x_n-x_{n-1}),
\end{array}\right.
\end{equation}
where $\a>0,\b\in(0,1)$ and $0<s<\frac{2(1-\b)}{L_g}.$

Note that \eqref{generaldiscrete} is a nonconvex descendant of the methods of Polyak \cite{poly} and Nesterov \cite{nesterov83}.
Indeed, in \cite{poly}, Polyak  introduced a modified  gradient method  for minimizing a smooth convex
function $g$. His two-step iterative method, the so called heavy ball method, takes the following form:
\begin{equation}\label{polyak}\left\{\begin{array}{lll}
\ds x_{n+1}=y_n-\l_n\n g(x_n),\\
\\
\ds y_n=x_n+\a_n(x_n-x_{n-1}),
\end{array}\right.
\end{equation}
where $\a_n\in [0, 1)$  and $\l_n>0$ is  a step-size parameter.

In a seminal paper \cite{nesterov83}, Nesterov proposed a modification of the heavy ball method in order to obtain optimal convergence rates
for smooth convex functions. More precisely, Nesterov used $\a_n=\frac{t_n-1}{t_{n+1}}$ where $t_n$ satisfies the recursion
$t_{n+1}=\frac{\sqrt{4t_n^2+1}+1}{2},\,t_1=1$  and put $y_n$ also for evaluating the gradient. Additionally, $\l_n$ is chosen in such way that $\l_n\le \frac{1}{L_g}.$ His  scheme in its simplest form is given by:
\begin{equation}\label{nest}\left\{\begin{array}{lll}
\ds x_{n+1}=y_n-s\n g(y_n),\\
\\
\ds y_n=x_n+\frac{t_n-1}{t_{n+1}}(x_n-x_{n-1}),
\end{array}\right.
\end{equation}
where $s\le \frac{1}{L_g}.$

This scheme leads to the convergence rate $g(x_n)-g(\ol x)=\mathcal{O}\left({1}/{n^2}\right)$, where $\ol x$ is a minimizer of the convex function $g$,  and this is optimal among all methods having only information about the gradient of $g$ and consecutive iterates, \cite{Nest}.

By taking $t_n=\frac{n+a-1}{a},\,a\ge2$ in \eqref{nest} we obtain an algorithm that is asymptotically equivalent to the original Nesterov method and leads to the same rate of convergence $\mathcal{O}\left({1}/{n^2}\right)$, (see \cite{su-boyd-candes, ch-do2015}). This case has been considered by Chambolle and Dossal \cite{ch-do2015}, in order to prove the convergence of the iterates of the modified FISTA algorithm (see \cite{bete09}). We emphasize that Algorithm \eqref{generaldiscrete} has a similar form as  the algorithm studied by Chambolle and Dossal (see \cite{ch-do2015} and also \cite{LP}), but we allow the function $g$ to be nonconvex.
 Unfortunately, our analysis do not cover the case  $\b=1.$

Su, Boyd and Cand\`es  (see  \cite{su-boyd-candes}), showed that in case $t_n=\frac{n+1}{2}$ the algorithm \eqref{nest} has the exact limit the second order differential equation
\begin{equation}\label{ee11}
\ddot{x}(t)+\frac{\a}{t}\dot{x}(t)+\n g(x(t))=0.
\end{equation}
with $\a=3.$

 Recently, Attouch and his co-authors (see \cite{att-c-p-r-math-pr2018,att-p-r-jde2016}), proved  that,  if $\a>3$ in \eqref{ee11},  then  the  generated  trajectory $x(t)$  converges  to  a  minimizer  of $g$ as $t\To+\infty$, while the convergence rate of the objective function along the trajectory is $o(1/t^2)$.  Further, in \cite{att-c-r-arx2017}, some results concerning the convergence rate of the objective function  $g$ along  the  trajectory generated by \eqref{ee11}, in  the  subcritical  case $\a\le 3$,  have been obtained. However, the convergence of the generated trajectories by \eqref{ee11} in case $g$ is nonconvex is still an open question.
 Some important steps in this direction have been made in \cite{BCL-AA} (see also \cite{BCL}), where convergence of the trajectories of a system, that can be viewed as a perturbation of \eqref{ee11}, have been obtained in a nonconvex setting. More precisely in \cite{BCL-AA}  is considered the system
  \begin{equation}\label{eee11}
\ddot{x}(t)+\left(\g+\frac{\a}{t}\right)\dot{x}(t)+\n g(x(t))=0,
\end{equation}
 and it is shown that the generated trajectory converges to a critical point of $g$, if a regularization of $g$ satisfies the Kurdyka-{\L}ojasiewicz property.

In what  follows we show that by choosing  appropriate values of $\b$, the numerical scheme \eqref{generaldiscrete}  has as the exact limit  the continuous second order dynamical systems \eqref{ee11} studied in \cite{su-boyd-candes, att-c-p-r-math-pr2018,att-p-r-jde2016,att-c-r-arx2017}, and also the continuous dynamical system \eqref{eee11} studied  in \cite{BCL-AA}. We take to this end in \eqref{generaldiscrete} small step sizes and follow the same approach as Su, Boyd and Cand\`es in \cite{su-boyd-candes}, (see also \cite{BCL-AA}). For this purpose we rewrite \eqref{generaldiscrete} in the form
\begin{equation}\label{e411}
\frac{x_{n+1}-x_n}{\sqrt{s}}=\frac{\b n}{n+\a}\cdot\frac{x_n-x_{n-1}}{\sqrt{s}}-\sqrt{s}\n g(y_n) \ \forall n \geq 1
\end{equation}
and introduce the {\it Ansatz} $x_n\approx x(n\sqrt{s})$ for some twice continuously differentiable function $x : [0,+\infty) \rightarrow \R^n$. We let $n=\frac{t}{\sqrt{s}}$ and get  $x(t)\approx x_n,\,x(t+\sqrt{s})\approx x_{n+1},\,x(t-\sqrt{s})\approx x_{n-1}.$
Then, as the step size $s$ goes to zero, from the Taylor expansion of $x$ we obtain
$$\frac{x_{n+1}-x_n}{\sqrt{s}}=\dot{x}(t)+\frac12\ddot{x}(t)\sqrt{s}+o(\sqrt{s})$$
and
$$\frac{x_n-x_{n-1}}{\sqrt{s}}=\dot{x}(t)-\frac12\ddot{x}(t)\sqrt{s}+o(\sqrt{s}).$$

Further, since
$$\sqrt{s}\|\n g(y_n)-\n g(x_n)\|\le\sqrt{s} L_g\|y_n-x_n\|=\sqrt{s} L_g\left|\frac{\b n}{n+\a}\right|\|x_n-x_{n-1}\|
=o(\sqrt{s}),$$
it follows $\sqrt{s} \n g(y_n)=\sqrt{s}\n g(x_n)+ o(\sqrt{s})$. Consequently, \eqref{e411} can be written as
$$\dot{x}(t)+\frac12\ddot{x}(t)\sqrt{s}+ o(\sqrt{s})= $$
$$\frac{\b t}{t+\a\sqrt{s}}\left(\dot{x}(t)-\frac12\ddot{x}(t)\sqrt{s}+ o(\sqrt{s})\right)-\sqrt{s}\n g(x(t))+ o(\sqrt{s})$$
or, equivalently
$$(t+\a\sqrt{s})\left(\dot{x}(t)+\frac12\ddot{x}(t)\sqrt{s}+o(\sqrt{s})\right)=$$
$$\b t\left(\dot{x}(t)-\frac12\ddot{x}(t)\sqrt{s}+o(\sqrt{s})\right)-\sqrt{s}(t+\a\sqrt{s})\n g(x(t))+o(\sqrt{s}).$$
Hence,
\begin{equation}\label{ecom}
\frac12\left(\a\sqrt{s}+(1+\b)t\right)\ddot{x}(t)\sqrt{s}+\left((1-\b)t+\a\sqrt{s}\right)\dot{x}(t)+\sqrt{s}(t+\a\sqrt{s})\n g(x(t))=o(\sqrt{s}).
\end{equation}
Now, if we take $\b=1-\g {s}<1$ in \eqref{ecom} for some $\frac{1}{{s}}>\g>0$, we obtain
$$\frac12\left(\a \sqrt{s}+(2-\g{s})t\right)\ddot{x}(t)\sqrt{s}+\left(\g{s}t+\a\sqrt{s}\right)\dot{x}(t)+\sqrt{s}(t+\a\sqrt{s})\n g(x(t))=o(\sqrt{s}).$$

After dividing by $\sqrt{s}$ and letting $s\rightarrow 0$, we obtain
$$t\ddot{x}(t)+\a\dot{x}(t)+t\n g(x(t))=0,$$
which, after division by $t$, gives \eqref{ee11}, that is
$$\ddot{x}(t)+\frac{\a}{t}\dot{x}(t)+\n g(x(t))=0.$$

Similarly, by taking   $\b=1-\g\sqrt{s}<1$ in \eqref{ecom}, for some $\frac{1}{\sqrt{s}}>\g>0$, we obtain
$$\frac12\left(\a \sqrt{s}+(2-\g\sqrt{s})t\right)\ddot{x}(t)\sqrt{s}+\left(\g\sqrt{s}t+\a\sqrt{s}\right)\dot{x}(t)+\sqrt{s}(t+\a\sqrt{s})\n g(x(t))=o(\sqrt{s}).$$
 After dividing by $\sqrt{s}$ and letting $s\rightarrow 0$, we get
$$t\ddot{x}(t)+(\g t+\a)\dot{x}(t)+t\n g(x(t))=0,$$
which, after division by $t$, gives \eqref{eee11}, that is
$$\ddot{x}(t)+\left(\g+\frac{\a}{t}\right)\dot{x}(t)+\n g(x(t))=0.$$

Consequently, our numerical scheme \eqref{generaldiscrete} can be seen as the discrete counterpart of the continuous dynamical systems \eqref{ee11} and \eqref{eee11}, in a full nonconvex setting.

The techniques for proving the convergence of \eqref{generaldiscrete} use the same main ingredients as other algorithms for nonconvex optimization problems involving KL functions. More precisely, in the next section, we show a sufficient decrease property for
the iterates, which also ensures that the iterates gap belongs to $l^2$, further we show that the set of cluster points of the iterates is included in the set of critical points of the  objective function, and, finally, we use the KL property of an appropriate regularization of the objective function in order to obtain  that the iterates gap belongs to $l^1$, which implies the  convergence of the iterates, see also \cite{att-b-sv2013,b-sab-teb,BCL1}.
Moreover, in  section 3, we obtain several convergence rates both for the sequences $(x_n)_{n\in\N},\,(y_n)_{n\in\N}$ generated by the numerical scheme \eqref{generaldiscrete}, as well as for the function values $g(x_n),\,g(y_n)$ in the terms of the {\L}ojasiewicz exponent of $g$ and a regularization of $g$, respectively (for some general results see  \cite{FGP}).

\section{The convergence of the generated sequences}
In this section we investigate the convergence of the proposed algorithm. We show that the sequences generated by the numerical scheme \eqref{generaldiscrete} converge  to a critical point of the objective function $g$, provided the regularization of $g$, $H(x,y)=g(x)+\frac12\|y-x\|^2,$ is a KL function.
The main tool in our forthcoming analysis is the so called descent lemma, see \cite{Nest}.

\begin{lemma}\label{desc} Let $g:\R^m\To\R$ be Fr\`echet differentiable with $L_g$ Lipschitz continuous gradient. Then
$$g(y)\le g(x)+\<\n g(x),y-x\>+\frac{L_g}{2}\|y-x\|^2,\,\forall x,y\in\R^m.$$
\end{lemma}

Now we are able to obtain a decrease property for the iterates generated by \eqref{generaldiscrete}.

\begin{theorem}\label{decreasing} In the settings of problem \eqref{opt-pb}, for some starting points $x_0=y_0\in\R^m$ let $(x_n)_{n\in\N},\,(y_n)_{n\in\N}$ be the sequences generated by the numerical scheme \eqref{generaldiscrete}. Consider the sequences
 $$A_{n-1}=\frac{2-s L_g}{2s}\left(\frac{(1+\b)n+\a}{n+\a}\right)^2-\frac{\b n((1+\b)n+\a)}{s(n+\a)^2},$$
$$C_{n-1}=\frac{2-s L_g}{2s}\frac{\b n-\b}{n+\a-1}\frac{(1+\b)n+\a}{n+\a}-\frac{1}{2s}\frac{\b n-\b}{n+\a-1}\frac{\b n}{n+\a}$$
and
$$\d_n=A_{n-1}-C_{n-1}$$
for all $n\in\N,\,n\ge 1$.

 Then, there exists $N\in\N$ such that 
\begin{itemize}
\item[(i)] The sequence $ \left(g(y_{n})+\d_{n}\|x_{n}-x_{n-1}\|^2\right)_{n\ge N}$ is decreasing and $\d_n>0$ for all $n\ge N$.
\end{itemize}
Assume that $g$ is bounded from below. Then, the following statements hold.
\begin{itemize}
\item[(ii)] The sequence $\left(g(y_{n})+\d_{n}\|x_{n}-x_{n-1}\|^2\right)_{n\in \N}$ is convergent;
\item[(iii)] $\sum_{n\ge 1}\|x_n-x_{n-1}\|^2<+\infty.$
\end{itemize}
\end{theorem}

\begin{proof}

From \eqref{generaldiscrete} we have $\n g(y_n)=\frac{1}{s}(y_n-x_{n+1})$, hence
$$\<\n g(y_n),y_{n+1}-y_n\>=\frac{1}{s}\<y_n-x_{n+1},y_{n+1}-y_n\>.$$
Now, from Lemma \ref{desc} we obtain
$$g(y_{n+1})\le g(y_n)+\<\n g(y_n),y_{n+1}-y_n\>+\frac{L_g}{2}\|y_{n+1}-y_n\|^2,$$
consequently we have
\begin{equation}\label{e1}
g(y_{n+1})-\frac{L_g}{2}\|y_{n+1}-y_n\|^2\le g(y_n)+ \frac{1}{s}\<y_n-x_{n+1},y_{n+1}-y_n\>.
\end{equation}

Further,
$$\<y_n-x_{n+1},y_{n+1}-y_n\>=-\|y_{n+1}-y_n\|^2+\<y_{n+1}-x_{n+1},y_{n+1}-y_n\>,$$
 and
$$y_{n+1}-x_{n+1}=\frac{\b(n+1)}{n+\a+1}(x_{n+1}-x_n),$$
hence
\begin{equation}\label{e2}
g(y_{n+1})+\left(\frac{1}{s}-\frac{L_g}{2}\right)\|y_{n+1}-y_n\|^2\le g(y_n)+ \frac{\frac{\b(n+1)}{n+\a+1}}{s}\<x_{n+1}-x_n,y_{n+1}-y_n\>.
\end{equation}

Since $$y_{n+1}-y_n=\frac{(1+\b)n+\a+\b+1}{n+\a+1}(x_{n+1}-x_n)-\frac{\b n}{n+\a}(x_n-x_{n-1}),$$
we have,
$$\|y_{n+1}-y_n\|^2=\left\|\frac{(1+\b)n+\a+\b+1}{n+\a+1}(x_{n+1}-x_n)-\frac{\b n}{n+\a}(x_n-x_{n-1})\right\|^2=$$
$$\left(\frac{(1+\b)n+\a+\b+1}{n+\a+1}\right)^2\|x_{n+1}-x_n\|^2+\left(\frac{\b n}{n+\a}\right)^2\|x_n-x_{n-1}\|^2-$$
$$2\frac{(1+\b)n+\a+\b+1}{n+\a+1}\frac{\b n}{n+\a}\<x_{n+1}-x_n,x_n-x_{n-1}\>,$$
and
$$\<x_{n+1}-x_n,y_{n+1}-y_n\>=\left\<x_{n+1}-x_n,\frac{(1+\b)n+\a+\b+1}{n+\a+1}(x_{n+1}-x_n)-\frac{\b n}{n+\a}(x_n-x_{n-1})\right\>=$$
$$\frac{(1+\b)n+\a+\b+1}{n+\a+1}\|x_{n+1}-x_n\|^2-\frac{\b n}{n+\a}\<x_{n+1}-x_n,x_n-x_{n-1}\>.$$

Replacing the above equalities in \eqref{e2}, we obtain

$$g(y_{n+1})+\frac{(2-s L_g)\left(\frac{(1+\b)n+\a+\b+1}{n+\a+1}\right)^2-2\frac{\b(n+1)((1+\b)n+\a+\b+1)}{(n+\a+1)^2}}{2s}\|x_{n+1}-x_n\|^2\le$$
$$g(y_n)-\frac{(2-s L_g)\left(\frac{\b n}{n+\a}\right)^2}{2s}\|x_{n}-x_{n-1}\|^2+$$
$$\frac{(2-s L_g)\frac{\b n}{n+\a}\frac{(1+\b)n+\a+\b+1}{n+\a+1}-\frac{\b n}{n+\a}\frac{\b(n+1)}{n+\a+1}}{s}\<x_{n+1}-x_n,x_n-x_{n-1}\>.$$

For simplicity let 
$$B_n=\frac{(2-s L_g)\left(\frac{\b n}{n+\a}\right)^2}{2s}$$
for all $n\in\N.$

 Hence we have
 $$g(y_{n+1})+A_{n}\|x_{n+1}-x_n\|^2-2C_n\<x_{n+1}-x_n,x_n-x_{n-1}\>\le g(y_n)-B_n\|x_{n}-x_{n-1}\|^2.$$

By using the equality
\begin{equation}\label{e5}
-2\<x_{n+1}-x_n,x_n-x_{n-1}\>=\|x_{n+1}+x_{n-1}-2x_n\|^2-\|x_{n+1}-x_n\|^2-\|x_n-x_{n-1}\|^2
\end{equation}
we obtain
$$g(y_{n+1})+(A_n-C_n)\|x_{n+1}-x_n\|^2+C_n\|x_{n+1}+x_{n-1}-2x_n\|^2\le g(y_n)+(C_n-B_n)\|x_n-x_{n-1}\|^2.$$

Note that $A_{n}-C_n=\d_{n+1}$ and let us denote $\Delta_n=B_n+A_{n-1}-C_{n-1}-C_n$. Consequently the following inequality holds.
\begin{equation}\label{e6}
C_n\|x_{n+1}+x_{n-1}-2x_n\|^2+\Delta_n\|x_n-x_{n-1}\|^2\le (g(y_n)+\d_n\|x_n-x_{n-1}\|^2)- (g(y_{n+1})+\d_{n+1}\|x_{n+1}-x_n\|^2).
\end{equation}

Since $0<\b<1$ and $s<\frac{2(1-\b)}{L_g},$  we have
$$\lim_{n\To+\infty}A_n=\frac{(2-s L_g)(\b+1)^2-2\b-2\b^2}{2s}>0,$$
$$\lim_{n\To+\infty}B_n=\frac{(2-s L_g)\b^2}{2s}>0,$$
$$\lim_{n\To+\infty}C_n= \frac{(2-s L_g)(\b^2+\b)-\b^2}{2s} >0,$$
$$\lim_{n\To+\infty}\Delta_n=\frac{2-s L_g-2\b}{2s}>0,$$
and
$$\lim_{n\To+\infty}\d_n=\frac{2-\b^2-s L_g(\b+1)}{2s}>0.$$

Hence, there exists $N\in\N$ and $C>0,\,D>0$ such that for all $n\ge N$ one has
$$C_n\ge C,\, \Delta_n\ge D\mbox{ and }\d_n>0$$ which, in the view of \eqref{e6}, shows (i), that is,  the sequence $g(y_n)+\d_n\|x_n-x_{n-1}\|^2$ is decreasing for $n\ge N.$

Assume now that $g$ is bounded from below. 
 By using \eqref{e6} again, we obtain
 $$0\le C\|x_{n+1}+x_{n-1}-2x_n\|^2+D\|x_n-x_{n-1}\|^2\le (g(y_n)+\d_n\|x_n-x_{n-1}\|^2)- (g(y_{n+1})+\d_{n+1}\|x_{n+1}-x_n\|^2),$$
for all $n\ge N,$ or more convenient, that
 \begin{equation}\label{ineq}
0\le D\|x_n-x_{n-1}\|^2\le (g(y_n)+\d_n\|x_n-x_{n-1}\|^2)- (g(y_{n+1})+\d_{n+1}\|x_{n+1}-x_n\|^2),
\end{equation}
for all $n\ge N.$ Let $r>N.$ By summing up the latter relation  we have
$$
D\sum_{n=N}^r\|x_n-x_{n-1}\|^2\le (g(y_N)+\d_N\|x_N-x_{N-1}\|^2)-(g(y_{r+1})+\d_{r+1}\|x_{r+1}-x_r\|^2)
$$
which leads to
\begin{equation}\label{forcoercive}
g(y_{r+1})+D\sum_{n=N}^r\|x_n-x_{n-1}\|^2\le g(y_N)+\d_N\|x_N-x_{N-1}\|^2.
\end{equation}
Now,  taking into account that $g$ is bounded from below, by letting $r\To+\infty$  we obtain
$$\sum_{n=N}^{\infty}\|x_n-x_{n-1}\|^2\le+\infty$$ which proves (iii).

The latter relation also shows that
$$\lim_{n\To+\infty}\|x_n-x_{n-1}\|^2=0,$$
hence $$\lim_{n\To+\infty}\d_n\|x_n-x_{n-1}\|^2=0.$$
But then, from the fact that $g$ is bounded from below we obtain that the sequence $g(y_n)+\d_n\|x_n-x_{n-1}\|^2$ is bounded from below. On the other hand, from (i) we have that  the sequence $g(y_n)+\d_n\|x_n-x_{n-1}\|^2$ is decreasing for $n\ge N,$ hence there exists
$$\lim_{n\To+\infty} g(y_n)+\d_n\|x_n-x_{n-1}\|^2\in\R.$$
\end{proof}

\begin{remark}\label{r1} Observe that conclusion (iii) in the hypotheses of Theorem \ref{decreasing} assures that the sequence  $(x_n-x_{n-1})_{n\in\N}\in l^2,$ in particular that
\begin{equation}\label{e7}
\lim_{n\To+\infty}(x_n-x_{n-1})=0.
\end{equation}
\end{remark}

Let us denote by $\omega((x_n)_{n\in\N})$ the set of cluster points of the sequence $(x_n)_{n\in\N},$ and denote by $\crit(g)=\{x\in\R^m: \nabla g(x)=0\}$ the set of critical points of $g$.

In the following result we use the distance function to a set, defined for $A\subseteq\R^n$ as $\dist(x,A)=\inf_{y\in A}\|x-y\|$ for all $x\in\R^n$.

\begin{lemma}\label{regularization} In the settings of problem \eqref{opt-pb}, for some starting points $x_0=y_0\in\R^m$ consider the sequences $(x_n)_{n\in\N},\,(y_n)_{n\in\N}$ generated by Algorithm \eqref{generaldiscrete}. Assume that $g$ is bounded from below and consider the function
$$H:\R^m\times\R^m\To\R,\,H(x,y)=g(x)+\frac12\|y-x\|^2.$$
Consider further the sequence
$$u_n=\sqrt{2\d_n}(x_{n}-x_{n-1})+y_n\mbox{, for all }n\in \N,$$
where $\d_n$ was defined in Theorem \ref{decreasing}. Then, the following statements hold true.
\begin{itemize}
\item[(i)] $\omega((u_n)_{n\in\N})=\omega((y_n)_{n\in\N})=\omega((x_n)_{n\in\N})\subseteq\crit g$;
\item[(ii)] There exists and is finite the limit $\lim_{n\To+\infty}H(y_n,u_n)$;
\item[(iii)] $\omega((y_n,u_n)_{n\in\N})\subseteq\crit H=\{(x,x)\in\R^m\times\R^m:x\in\crit g\}$;
\item[(iv)] $\|\n H(y_n,u_n)\|\le \frac{1}{s}\|x_{n+1}-x_n\|+\left(\frac{\b  n}{s(n+\a)}+2\sqrt{2\d_n}\right)\|x_{n}-x_{n-1}\|$ for all $n\in\N$;
\item[(v)]  $\|\n H(y_n,u_n)\|^2\le \frac{2}{s^2}\|x_{n+1}-x_n\|^2+2\left(\left(\frac{\b n}{s(n+\a)}-\sqrt{2\d_n}\right)^2+\d_n\right)\|x_{n}-x_{n-1}\|^2$ for all $n\in\N$;
\item[(vi)] $H$ is finite and constant on $\omega((y_n,u_n)_{n\in\N}).$
\end{itemize}

Assume that $(x_n)_{n\in\N}$ is bounded. Then,
\begin{itemize}
\item[(vii)] $\omega((y_n, u_n)_{n\in\N})$ is nonempty and compact;
\item[(viii)] $\lim_{n\To+\infty} \dist((y_n,u_n),\omega((y_n,u_n)_{n\in\N}))=0.$
\end{itemize}
\end{lemma}
\begin{proof}
(i) Let $\ol x\in\omega((x_n)_{n\in\N}).$ Then, there exists a subsequence $(x_{n_k})_{k\in \N}$ of $(x_n)_{n\in\N}$ such that
$$\lim_{k\to+\infty}x_{n_k}=\ol x.$$
Since by \eqref{e7} $\lim_{n\To+\infty}(x_n-x_{n-1})=0$ and the sequences $(\sqrt{2\d_n})_{n\in\N},\,\left(\frac{\b n}{n+\a}\right)_{n\in\N}$ converge, we obtain that
$$\lim_{k\to+\infty}y_{n_k}=\lim_{k\to+\infty}u_{n_k}=\lim_{k\to+\infty}x_{n_k}=\ol x,$$
which shows that $$\omega((x_n)_{n\in\N})\subseteq \omega((u_n)_{n\in\N})\mbox{ and }\omega((x_n)_{n\in\N})\subseteq \omega((y_n)_{n\in\N}).$$
Further from \eqref{generaldiscrete}, the continuity of $\n g$ and \eqref{e7}, we obtain that
$$\n g(\ol x)=\lim_{k\To +\infty} \n g(y_{n_k})=\frac{1}{s}\lim_{k\To +\infty}(y_{n_k}-x_{n_k+1})=$$
$$\frac{1}{s}\lim_{k\To +\infty}\left[(x_{n_k}-x_{n_k+1})+\frac{\b n_k}{n_k+\a}(x_{n_k}-x_{n_k-1})\right]=0.$$
Hence $\omega((x_n)_{n\in\N})\subseteq\crit g.$
Conversely, if $\ol u\in \omega((u_n)_{n\in\N})$ then, from \eqref{e7} results that $\ol u\in \omega((y_n)_{n\in\N})$ and $\ol u\in \omega((x_n)_{n\in\N}).$
Hence,
$$\omega((u_n)_{n\in\N})=\omega((y_n)_{n\in\N})=\omega((x_n)_{n\in\N})\subseteq\crit g.$$

(ii) is nothing else than (ii) in Theorem \ref{decreasing}.

For (iii) observe that  $\n H(x,y)=(\n g(x)+x-y,y-x)$, hence, $\n H(x,y)=0$ leads to
$x=y$ and $\n g(x)=0.$
Consequently
$$\crit H=\{(x,x)\in\R^m\times\R^m: x\in\crit g\}.$$

Further, consider $(\ol y,\ol u)\in \omega((y_n,u_n)_{n\in\N}).$ Then, there exists $(y_{n_k},u_{n_k})_{k\in\N}\subseteq (y_n,u_n)_{n\in\N}$ such that
$$(\ol y,\ol u)=\lim_{k\To+\infty}(y_{n_k},u_{n_k})=\lim_{k\To+\infty}(x_{n_k},x_{n_k})=(\ol x,\ol x).$$
Hence, $\ol u=\ol y=\ol x\in\omega((x_n)_{n\in\N})\subseteq \crit g$ and $(\ol x,\ol x)\in\crit H.$

(iv) By using the 1-norm of $\R^m\times\R^m$ and \eqref{generaldiscrete}, for every $n\in \N$ we have
 $$\|\n H(y_n,u_n)\|\le\|\n H(y_n,u_n)\|_1=\|(\n g(y_n)+y_n-u_n,u_n-y_n)\|_1=\|\n g(y_n)+y_n-u_n\|+\|u_n-y_n\|\le$$
 $$\|\n g(y_n)\|+2\|\sqrt{2\d_n}(x_{n}-x_{n-1})\|=$$
 $$\frac{1}{s}\left\|\left(x_n+\frac{\b n}{n+\a}(x_n-x_{n-1})\right)-x_{n+1}\right\|+2\sqrt{2\d_n}\|x_{n}-x_{n-1}\|\le$$
 $$\frac{1}{s}\|x_{n+1}-x_n\|+\left(\frac{\b  n}{s(n+\a)}+2\sqrt{2\d_n}\right)\|x_{n}-x_{n-1}\|.$$

(v) We use the euclidian norm of $\R^m\times\R^m$, that is $\|(x,y)\|=\sqrt{\|x\|^2+\|y\|^2}$ for all $(x,y)\in\R^m\times\R^m.$
We have:
$$\|\n H(y_n,u_n)\|^2=\|(\n g(y_n)+y_n-u_n,u_n-y_n)\|^2=\|\n g(y_n)+y_n-u_n\|^2+\|u_n-y_n\|^2=$$
$$\left\|\frac1s(x_n-x_{n+1})+\left(\frac{\b n}{s(n+\a)}-\sqrt{2\d_n}\right)(x_n-x_{n-1})\right\|^2+2\d_n\|x_n-x_{n-1}\|^2\le$$
$$\frac{2}{s^2}\|x_{n+1}-x_n\|^2+2\left(\left(\frac{\b n}{s(n+\a)}-\sqrt{2\d_n}\right)^2+\d_n\right)\|x_{n}-x_{n-1}\|^2$$ for all $n\in\N.$

(vi) follows directly from (ii).

 Assume now that $(x_n)_{n\in\N}$ is bounded and let us prove (vii), (see also \cite{BCL-AA}). Obviously $(y_n,u_n)_{n\in\N}$ is bounded, hence according to Weierstrass Theorem $\omega((y_n, u_n)_{n\in\N}),$ (and also $\omega((x_n)_{n\in\N})$), is nonempty. It remains to show that $\omega((y_n, u_n)_{n\in\N})$ is closed. From (i) and the proof of (iii) we have
 \begin{equation}\label{e8}
\omega((y_n, u_n)_{n\in\N})=\{(\ol x,\ol x)\in\R^m\times \R^m: \ol x\in \omega((x_n)_{n\in\N})\}.
\end{equation}
 Hence, it is enough to show that $\omega((x_n)_{n\in\N})$ is closed.

 Let be $(\ol x_p)_{p\in\N}\subseteq \omega((x_n)_{n\in\N})$ and assume that $\lim_{p\To+\infty}\ol x_p=x^*.$ We show that $x^*\in\omega((x_n)_{n\in\N}).$ Obviously, for every $p\in\N$ there exists a sequence of natural numbers $n_k^p\To+\infty,\,k\To+\infty$, such that
$$\lim_{k\To+\infty}x_{n_k^p}=\ol x_p.$$

Let be $\e >0$. Since $\lim_{p\To+\infty}\ol x_p=x^*,$  there exists $P(\e) \in \N$ such that for every $p\ge P(\e)$ it holds
$$\|\ol x_p-x^*\|<\frac\e2.$$
Let $p\in\N$ be fixed. Since $\lim_{k\To+\infty}x_{n_k^p}=\ol x_p,$ there exists $k(p,\e)\in\N$ such that for every $k\ge k(p,\e)$ it holds $$\|x_{n_k^p}-\ol x_p\|<\frac\e2.$$
Let be $k_p \geq k(p,\varepsilon)$ such that $n_{k_p}^p >p$. Obviously $n_{k_p}^p \To \infty$ as $p \To +\infty$ and for every $p \geq P(\e)$
$$\|x_{n_{k_p}^p}-x^*\|<\e.$$
Hence
$$\lim_{p\To+\infty}x_{n_{k_p}^p}=x^*,$$
thus $x^*\in \omega((x_n)_{n\in\N}).$

(viii) By using \eqref{e8} we have
$$\lim_{n\To+\infty} \dist((y_n,u_n),\omega((y_n,u_n)_{n\in\N}))=\lim_{n\To+\infty}\inf_{\ol x\in\omega((x_n)_{n\in\N})}\|(y_n,u_n)-(\ol x,\ol x)\|.$$
Since there exists the subsequences $(y_{n_k})_{k\in\N}$ and $(u_{n_k})_{k\in\N}$ such that
$\lim_{k\To\infty}y_{n_k}=\lim_{k\To\infty}u_{n_k}=\ol x_0\in \omega((x_n)_{n\in\N})$ it is straightforward that
$$\lim_{n\To+\infty} \dist((y_n,u_n),\omega((y_n,u_n)_{n\in\N}))=0.$$
\end{proof}

\begin{remark}\label{r2} We emphasize  that if $g$ is coercive, that is
$\lim_{\|x\|\rightarrow+\infty} g(x)=+\infty,$
then  $g$ is bounded from below and $(x_n)_{n\in\N},\,(y_n)_{n\in\N}$, the sequences generated by \eqref{generaldiscrete}, are bounded.

Indeed, notice that $g$ is bounded from below, being a continuous and coercive function (see \cite{rock-wets}). Note that according to Theorem \ref{decreasing} the sequence $D\sum_{n=N}^r\|x_n-x_{n-1}\|^2$ is convergent hence is bounded. Consequently, from \eqref{forcoercive} it follows that $y_r$ is contained for every $r> N,$ ($N$ is defined in the hypothesis of Theorem \ref{decreasing}), in a lower level set of $g$, which is bounded. Since $(y_n)_{n\in\N}$ is bounded, taking into account \eqref{e7}, it follows that $(x_n)_{n\in\N}$  is also bounded.
\end{remark}

In order to continue our analysis we need the concept of a KL function. For $\eta\in(0,+\infty]$, we denote by $\Theta_{\eta}$
the class of concave and continuous functions $\varphi:[0,\eta)\rightarrow [0,+\infty)$ such that $\varphi(0)=0$, $\varphi$ is
continuously differentiable on $(0,\eta)$, continuous at $0$ and $\varphi'(s)>0$ for all
$s\in(0, \eta)$.

\begin{definition}\label{KL-property} \rm({\it Kurdyka-\L{}ojasiewicz property}) Let $f:\R^n\rightarrow \R$ be a differentiable function. We say that $f$ satisfies the {\it Kurdyka-\L{}ojasiewicz (KL) property} at
$\ol x\in \R^n$
if there exist $\eta \in(0,+\infty]$, a neighborhood $U$ of $\ol x$ and a function $\varphi\in \Theta_{\eta}$ such that for all $x$ in the
intersection
$$U\cap \{x\in\R^n: f(\ol x)<f(x)<f(\ol x)+\eta\}$$ the following inequality holds
$$\varphi'(f(x)-f(\ol x))\|\n f(x))\|\geq 1.$$
If $f$ satisfies the KL property at each point in $\R^n$, then $f$ is called a {\it KL function}.
\end{definition}

The origins of this notion go back to the pioneering work of \L{}ojasiewicz \cite{lojasiewicz1963}, where it is proved that for a
real-analytic function $f:\R^n\rightarrow\R$ and a critical point $\ol x\in\R^n$ (that is $\nabla f(\ol x)=0$), there exists $\theta\in[1/2,1)$ such that the function
$|f-f(\ol x)|^{\theta}\|\nabla f\|^{-1}$ is bounded around $\ol x$. This corresponds to the situation when
$\varphi(s)=C(1-\theta)^{-1}s^{1-\theta}$. The result of \L{}ojasiewicz allows the interpretation of the KL property as a re-parametrization of the function values in order to avoid flatness around the
critical points. Kurdyka \cite{kurdyka1998} extended this property to differentiable functions definable in an o-minimal structure.
Further extensions to the nonsmooth setting can be found in \cite{b-d-l2006, att-b-red-soub2010, b-d-l-s2007, b-d-l-m2010}.

To the class of KL functions belong semi-algebraic, real sub-analytic, semiconvex, uniformly convex and
convex functions satisfying a growth condition. We refer the reader to
\cite{b-d-l2006, att-b-red-soub2010, b-d-l-m2010, b-sab-teb, b-d-l-s2007, att-b-sv2013, attouch-bolte2009} and the references therein  for more details regarding all the classes mentioned above and illustrating examples.

An important role in our convergence analysis will be played by the following uniformized KL property given in \cite[Lemma 6]{b-sab-teb}.

\begin{lemma}\label{unif-KL-property} Let $\Omega\subseteq \R^n$ be a compact set and let $f:\R^n\rightarrow \R$ be a differentiable function. Assume that $f$ is constant on $\Omega$ and $f$ satisfies the KL property at each point of $\Omega$.
Then there exist $\varepsilon,\eta >0$ and $\varphi\in \Theta_{\eta}$ such that for all $\ol x\in\Omega$ and for all $x$ in the intersection
\begin{equation}\label{int} \{x\in\R^n: \dist(x,\Omega)<\varepsilon\}\cap \{x\in\R^n: f(\ol x)<f(x)<f(\ol x)+\eta\}\end{equation}
the following inequality holds \begin{equation}\label{KL-ineq}\varphi'(f(x)-f(\ol x))\|\n f(x)\|\geq 1.\end{equation}
\end{lemma}

The following convergence result is the  first main result of the paper.

\begin{theorem}\label{convergence} In the settings of problem \eqref{opt-pb}, for some starting points $x_0=y_0\in\R^m$ consider the sequences $(x_n)_{n\in\N},\,(y_n)_{n\in\N}$ generated by Algorithm \eqref{generaldiscrete}. Assume that $g$ is bounded from below and consider the function
$$H:\R^m\times\R^m\To\R,\,H(x,y)=g(x)+\frac12\|y-x\|^2.$$
Assume that $(x_n)_{n\in\N}$ is bounded and $H$ is a KL function. Then the following statements are true
\begin{itemize}
\item[(a)] $\sum_{n\ge 1}\|x_n-x_{n-1}\|<+\infty;$
\item[(b)] there exists ${x}\in\crit(g)$ such that $\lim_{n\To+\infty}x_n={x}.$
\end{itemize}
\end{theorem}
\begin{proof} Consider the sequence
$$u_n=\sqrt{2\d_n}(x_{n}-x_{n-1})+y_n\mbox{, for all }n\in \N,$$
that  was defined in the hypotheses of Lemma \ref{regularization}. Furthermore, consider $(\ol x,\ol x)\in\omega((y_n,u_n)_{n\in\N}).$

Then, according to Lemma \ref{regularization}, the sequence $H(y_n,u_n)$ is decreasing for all $n\ge N$,  where $N$ was defined in Theorem \ref{decreasing}, further
$$\ol x\in\crit g\mbox{ and  }\lim_{n\To+\infty}H(y_n,u_n)=H(\ol x,\ol x).$$
We divide the proof into two cases.

{\bf Case I.} There exists $\ol n\ge N,\,\ol n\in\N,$  such that $H(y_{\ol n},u_{\ol n})=H(\ol x,\ol x).$ Then, since $H(y_n,u_n)$ is decreasing for all $n\ge N$ and $\lim_{n\To+\infty}H(y_n,u_n)=H(\ol x,\ol x)$ we obtain that
$$H(y_n,u_n)=H(\ol x,\ol x)\mbox{ for all }n\ge\ol n.$$

The latter relation combined with \eqref{ineq} leads to
$$0\le D\|x_n-x_{n-1}\|^2\le H(y_n,u_n)-H(y_{n+1},u_{n+1})=H(\ol x,\ol x)-H(\ol x,\ol x)=0$$
for all $n\ge\ol n.$

Hence $(x_n)_{n\ge\ol n}$ is constant and the conclusion follows.

{\bf Case II.} For every $n\ge N$ one has that $H(y_n,u_n)> H(\ol x,\ol x).$
Let $\Omega=\omega((y_n,u_n)_{n\in\N}).$ Then according to Lemma \ref{regularization}, $\Omega$ is nonempty and compact and $H$ is constant on $\Omega.$
Since $H$ is KL, according to Lemma \ref{unif-KL-property}  there exist $\varepsilon,\eta >0$ and $\varphi\in \Theta_{\eta}$ such that for all  $(z,w)$ belonging to the intersection
$$ \{(z,w)\in\R^n\times\R^n: \dist((z,w),\Omega)<\varepsilon\}\cap \{(z,w)\in\R^n\times\R^n: H(\ol x,\ol x)<H(z,w)<H(\ol x,\ol x)+\eta\}$$
one has
$$\varphi'(H(z,w)-H(\ol x,\ol x))\|\n H(z,w)\|\geq 1.$$

Since $\lim_{n\To+\infty}\dist((y_n,u_n),\Omega)=0$, there exists $n_1\in\N$ such that
$$\dist((y_n,u_n),\Omega)<\e,\,\forall n\ge n_1.$$

Since $$\lim_{n\To+\infty}H(y_n,u_n)=H(\ol x,\ol x)$$
and
$$H(y_n,u_n))> H(\ol{x},\ol{x})\mbox{ for all }n\ge N,$$
there exists $n_2\ge N$ such that
$$H(\ol x,\ol x)<H(y_n,u_n)< H(\ol{x},\ol{x})+\eta,\,\forall n\ge n_2.$$

Hence, for all $n\ge \ol n=\max(n_1,n_2)$ we have
$$\varphi'(H(y_n,u_n)-H(\ol{x},\ol{x}))\cdot\|\n H(y_n,u_n))\|\ge 1.$$
Since $\varphi$ is concave, for all $n\in\N$ we have
$$\varphi(H(y_n,u_n)-H(\ol{x},\ol{x}))-\varphi(H(y_{n+1},u_{n+1})-H(\ol{x},\ol{x}))\ge$$
$$\varphi'(H(y_n,u_n)-H(\ol{x},\ol{x}))\cdot(H(y_n,u_n)-H(y_{n+1},u_{n+1})),$$
hence,
$$\varphi(H(y_n,u_n)-H(\ol{x},\ol{x}))-\varphi(H(y_{n+1},u_{n+1})-H(\ol{x},\ol{x}))\ge \frac{H(y_n,u_n)-H(y_{n+1},u_{n+1})}{\|\n H(y_n,u_n)\|}$$
for all $n\ge\ol n.$

 Now, from \eqref{ineq} and Lemma \ref{regularization} (iv) we obtain
 \begin{equation}\label{e9}
 \varphi(H(y_n,u_n)-H(\ol{x},\ol{x}))-\varphi(H(y_{n+1},u_{n+1})-H(\ol{x},\ol{x}))\ge
 \end{equation}
 $$\frac{D\|x_{n}-x_{n-1}\|^2}{\frac{1}{s}\|x_{n+1}-x_n\|+\left(\frac{\b  n}{s(n+\a)}+2\sqrt{2\d_n}\right)\|x_{n}-x_{n-1}\|},$$
 for all $n\ge\ol n.$
Since $\lim_{n\To+\infty}\d_n=\frac{2-\b^2-s L_g(\b+1)}{2s}>0$ and $\lim_{n\To+\infty}\frac{\b  n}{s(n+\a)}=\frac{\b }{s}\ge 0$ there exists $\ol N\in\N,\,\ol N\ge\ol n$ and $M>0$ such that
$$\max\left(\frac1s,\frac{\b  n}{s(n+\a)}+2\sqrt{2\d_n}\right)\le M,$$
for all $n\ge \ol N.$
Hence, \eqref{e9} becomes
 \begin{equation}\label{e10}
 \varphi(H(y_n,u_n)-H(\ol{x},\ol{x}))-\varphi(H(y_{n+1},u_{n+1})-H(\ol{x},\ol{x}))\ge
 \end{equation}
 $$\frac{D\|x_{n}-x_{n-1}\|^2}{M(\|x_{n+1}-x_n\|+\|x_{n}-x_{n-1}\|)},$$
 for all $n\ge\ol N.$

Consequently,
$$\|x_{n}-x_{n-1}\|\le$$
$$\sqrt{\frac{M}{D}(\varphi(H(y_n,u_n)-H(\ol{x},\ol{x}))-\varphi(H(y_{n+1},u_{n+1})-H(\ol{x},\ol{x})))\cdot(\|x_{n+1}-x_n\|+\|x_{n}-x_{n-1}\|)},$$
 for all $n\ge\ol N.$

By using the arithmetical-geometrical mean inequality we have
$$\sqrt{\frac{M}{D}(\varphi(H(y_n,u_n)-H(\ol{x},\ol{x}))-\varphi(H(y_{n+1},u_{n+1})-H(\ol{x},\ol{x})))\cdot(\|x_{n+1}-x_n\|+\|x_{n}-x_{n-1}\|)}\le$$
$$\frac{\|x_{n+1}-x_n\|+\|x_{n}-x_{n-1}\|}{3}+\frac{3M}{4D}(\varphi(H(y_n,u_n)-H(\ol{x},\ol{x}))-\varphi(H(y_{n+1},u_{n+1})-H(\ol{x},\ol{x})))$$
 for all $n\ge\ol N.$
Hence
$$\|x_{n}-x_{n-1}\|\le$$
$$\frac{\|x_{n+1}-x_n\|+\|x_{n}-x_{n-1}\|}{3}+\frac{3M}{4D}(\varphi(H(y_n,u_n)-H(\ol{x},\ol{x}))-\varphi(H(y_{n+1},u_{n+1})-H(\ol{x},\ol{x})))$$
 for all $n\ge\ol N$ which leads to
 \begin{equation}\label{e11}
 2\|x_{n}-x_{n-1}\|-\|x_{n+1}-x_{n}\|\le\frac{9M}{4D}(\varphi(H(y_n,u_n)-H(\ol{x},\ol{x}))-\varphi(H(y_{n+1},u_{n+1})-H(\ol{x},\ol{x})))
 \end{equation}
 for all $n\ge\ol N.$ Let $P>\ol N$. By summing up \eqref{e11} from $\ol N$ to $P$ we obtain
 $$\sum_{n={\ol N}}^P \|x_{n}-x_{n-1}\|\le$$
  $$-\|x_{\ol N}-x_{\ol N-1}\|+\|x_{P+1}-x_P\|+\frac{9M}{4D}(\varphi(H(y_{\ol N},u_{\ol N})-H(\ol{x},\ol{x}))-\varphi(H(y_{P+1},u_{P+1})-H(\ol{x},\ol{x}))).$$
Now, by letting $P\To+\infty$ and using the fact that $\varphi(0)=0$ and \eqref{e7} we obtain that
$$\sum_{n={\ol N}}^\infty \|x_{n}-x_{n-1}\|\le-\|x_{\ol N}-x_{\ol N-1}\|+\frac{9M}{4D}\varphi(H(y_{\ol N},u_{\ol N})-H(\ol{x},\ol{x}))<+\infty,$$
hence
$$\sum_{n\ge 1}\|x_n-x_{n-1}\|<+\infty$$ which is exactly (a).

Obviously the sequence $S_n=\sum_{k=1}^n\|x_k-x_{k-1}\|$ is Cauchy, hence, for all $\e>0$ there exists $N_{\e}\in \N$   such that for all $n\ge N_\e$ and for all $p\in \N$ one has
$$S_{n+p}-S_n\le \e.$$
But
$$S_{n+p}-S_n=\sum_{k={n+1}}^{n+p}\|x_k-x_{k-1}\|\ge \left\|\sum_{k={n+1}}^{n+p}(x_k-x_{k-1})\right\|=\|x_{n+p}-x_n\|$$
hence the sequence $(x_n)_{n\in\N}$ is Cauchy, consequently is convergent. Let
$$\lim_{n\To+\infty}x_n=x.$$
Now, according to Lemma \ref{regularization} (i) one has
$$\{x\}=\omega((x_n)_{n\in\N})\subseteq\crit g$$ which proves (b).
\end{proof}

\begin{remark} Since the class of semi-algebraic functions is closed under addition (see for example \cite{b-sab-teb}) and
$(x,y) \mapsto \frac12\|x-y\|^2$ is semi-algebraic, the conclusion of the
previous theorem holds if the condition $H$ is a KL function is replaced by the assumption that $g$ is semi-algebraic.
\end{remark}

\begin{remark} Note that, according to Remark \ref{r2}, the conclusion of Theorem \ref{convergence} remains valid if we replace in its hypotheses the conditions that $g$ is bounded from below and $(x_n)_{n\in\N}$ is bounded by the condition that $g$ is coercive.
\end{remark}

\begin{remark} Note that under the assumptions of Theorem \ref{convergence} we have $\lim_{n\To+\infty}y_n= x$ and
$$\lim_{n\To+\infty}g(x_n)=\lim_{n\To+\infty}g(y_n)=g(x).$$
\end{remark}

\section{Convergence rates}\label{sec5}

In this section we will assume that the regularized function $H$ satisfies the Lojasiewicz property, which, as noted in the previous section, corresponds to a particular choice of the desingularizing function $\varphi$ (see \cite{lojasiewicz1963, b-d-l2006, attouch-bolte2009}).

\begin{definition}\label{Lexpo}  Let $f:\R^n\To \R $ be a differentiable function. The function $f$ is said to fulfill the {\L}ojasiewicz property, if for every $\ol x\in\crit{f}$ there exist $K,\e>0$ and $\t\in(0,1)$ such that
$$|f(x)-f(\ol x)|^\t\le K\|\n f(x)\|\mbox{ for every }x\mbox{ fulfilling }\|x-\ol x\|<\e.$$
The number $\t$  is called the {\L}ojasiewicz exponent of $f$ at the critical  point $\ol x.$ This corresponds to the case when the desingularizing function $\varphi$ has the form $\varphi(t)=\frac{K}{1-\t} t^{1-\t}.$
\end{definition}

In the following theorems we provide convergence rates for the sequence generated by \eqref{generaldiscrete}, but also for the function values, in terms of the {\L}ojasiewicz exponent of $H$ (see, also, \cite{b-d-l2006, attouch-bolte2009}). Note that the forthcoming results remain valid if one  replace in their hypotheses the conditions that $g$ is bounded from below and $(x_n)_{n\in\N}$ is bounded by the condition that $g$ is coercive.

\begin{theorem}\label{th-conv-rate}  In the settings of problem \eqref{opt-pb} consider the sequences $(x_n)_{n\in\N},\,(y_n)_{n\in\N}$ generated by Algorithm \eqref{generaldiscrete}. Assume that $g$ is bounded from below and that $(x_n)_{n\in\N}$ is bounded, let $\ol x\in\crit(g)$ be such that $\lim_{n\To+\infty}x_n=\ol x$ and suppose that
$$H:\R^n\times \R^n\To\R,\,H(x,y)=g(x)+\frac{1}{2}\|x-y\|^2$$
fulfills the {\L}ojasiewicz property at $(\ol x,\ol x)\in\crit H$ with {\L}ojasiewicz exponent  $\t\in\left(0,\frac12\right].$ Then, for every $p>0$ there exist $a_1,a_2,a_3,a_4>0$ and $\ol k\in\N$ such that the following statements hold true:
\begin{itemize}
\item[$(\emph{a}_1)$] $g(y_{n})-g(\ol x)\le a_1 \frac{1}{n^p}$ for  every $n> \ol k$,
\item[$(\emph{a}_2)$] $g(x_{n})-g(\ol x)\le a_2\frac{1}{n^p}$ for  every $n>\ol k$,
\item[$(\emph{a}_3)$] $\|x_n-\ol x\|\le a_3 \frac{1}{n^{\frac{p}{2}}}$ for  every $n>\ol k$,
\item[$(\emph{a}_4)$] $\|y_n-\ol x\|\le a_4 \frac{1}{n^{\frac{p}{2}}}$ for all  $n>\ol k$.
\end{itemize}
\end{theorem}
\begin{proof}
 As we have seen in the proof of Theorem \ref{convergence},  if there exists $\ol n\ge N,\,\ol n\in\N$, (where $N$ was defined in Theorem \ref{decreasing}), such that $H(y_{\ol n},u_{\ol n})=H(\ol x,\ol x),$ then,
$$H(y_n,u_n)=H(\ol x,\ol x)\mbox{ for all }n\ge\ol n$$
and $(x_n)_{n\ge\ol n}$ is constant. Consequently $(y_n)_{n\ge\ol n}$ is also constant and the conclusion of the theorem is straightforward.

Hence, in what follows we assume  that $H(y_{n},u_{ n})>H(\ol x,\ol x),$ for all $n\ge N.$
\\

Let us fix $p>0$ and let us  prove (a$_1$).

 For simplicity let us denote $r_n=H(y_n,u_n)-H(\ol x,\ol x)>0$  for all $n\in\N.$
 From \eqref{e6} we have
$$\Delta_n\|x_n-x_{n-1}\|^2\le r_n-r_{n+1}\mbox{ for all }n\ge N.$$

From Lemma \ref{regularization} (v) we have
 $$\|\n H(y_n,u_n)\|^2\le \frac{2}{s^2}\|x_{n+1}-x_n\|^2+2\left(\left(\frac{\b n}{s(n+\a)}-\sqrt{2\d_n}\right)^2+\d_n\right)\|x_{n}-x_{n-1}\|^2$$ for all $n\in\N$.
  Let $S_n=2\left(\left(\frac{\b n}{s(n+\a)}-\sqrt{2\d_n}\right)^2+\d_n\right),$ for all $n\in\N.$

It follows that, for all $n\ge N$ one has
$$\|x_n-x_{n-1}\|^2\ge \frac{1}{S_n}\|\n H(y_n,u_n)\|^2-\frac{2}{s^2S_n}\|x_{n+1}-x_n\|^2\ge$$
$$\frac{1}{S_n}\|\n H(y_n,u_n)\|^2-\frac{2}{s^2S_n\Delta_{n+1}}(r_{n+1}-r_{n+2}).$$

Now by using the {\L}ojasiewicz property of $H$ at $(\ol x,\ol x)\in\crit H$, and the fact that $\lim_{n\To+\infty}(y_n,u_n)=(\ol x,\ol x)$,  we obtain that there exists  $K,\e>0$ and $\ol N_1\in\N,$ such that for all $n\ge \ol N_1$ one has
$$\|(y_n,u_n)-(\ol x,\ol x)\|<\e,$$
consequently
\begin{equation}\label{e111} r_n-r_{n+1}\ge \frac{\Delta_n}{S_n}\|\n H(y_n,u_n)\|^2-\frac{2\Delta_n}{s^2S_n\Delta_{n+1}}(r_{n+1}-r_{n+2})\ge
\end{equation}
$$\frac{\Delta_n}{K^2 S_n}r_n^{2\t}-\frac{2\Delta_n}{s^2S_n\Delta_{n+1}}(r_{n+1}-r_{n+2})\ge$$
$$\frac{\Delta_n}{K^2 S_n}r_{n+1}^{2\t}-\frac{2\Delta_n}{s^2S_n\Delta_{n+1}}(r_{n+1}-r_{n+2})=\a_n r_{n+1}^{2\t}-\b_n(r_{n+1}-r_{n+2}),$$
where $\a_n= \frac{\Delta_n}{K^2 S_n}\mbox{ and  }\b_n=\frac{2\Delta_n}{s^2S_n\Delta_{n+1}}.$

It is obvious that the sequences $(\a_n)_{n\ge\ol N_1}$ and $(\b_n)_{n\ge\ol N_1}$ are convergent, further
$$\lim_{n\To+\infty} \a_n>0\mbox{ and } \lim_{n\To+\infty} \b_n>0.$$
Now, since $0<2\t\le1$ and $r_{n+1}\To 0$, there exists $\ol N_2\in\N,\, \ol N_2\ge \ol N_1,$ such that $r_{n+1}^{2\t}\ge r_{n+1}$ for all $n\ge\ol N_2.$

$(*)$\,\,\, Note that this implies that $0\le r_n\le 1$ for all $n\ge\ol N_2.$

Hence,
$$r_n\ge \left(\a_n-\b_n+1\right)r_{n+1}+\b_nr_{n+2},\mbox{ for all }n\ge\ol N_2.$$

Let us define for every $n>\ol N_2$ the sequence $\Xi_n=\frac{\b_n n^p}{(n+1)^p-n^p}.$
Then, since $p>0$ one has $\lim_{n\To+\infty}\Xi_n=+\infty.$

Since
$$\lim_{n\To+\infty}\left(\a_n-\b_n+1\right)-\left(1+\frac{\b_{n+1}}{\Xi_{n+1}}-\frac{\b_{n-1}}{1+\frac{\b_{n}}{\Xi_{n}}}\right)=\lim_{n\To+\infty}\a_n>0,$$  there exists $\ol k\in\N,\,\ol k\ge\ol N_2$ such that for all $n\ge\ol k$ one has
$$\a_n-\b_n+1\ge 1+\frac{\b_{n+1}}{\Xi_{n+1}}-\frac{\b_{n-1}}{1+\frac{\b_{n}}{\Xi_{n}}}.$$

Consequently,
$$r_n\ge \left(1+\frac{\b_{n+1}}{\Xi_{n+1}}-\frac{\b_{n-1}}{1+\frac{\b_{n}}{\Xi_{n}}}\right)r_{n+1}+\b_nr_{n+2},\mbox{ for all }n\ge\ol k,$$
or, equivalently
\begin{equation}\label{e12}
r_n+\frac{\b_{n-1}}{1+\frac{\b_{n}}{\Xi_{n}}}r_{n+1}\ge \left(1+\frac{\b_{n+1}}{\Xi_{n+1}}\right)\left(r_{n+1}+\frac{\b_n}{1+\frac{\b_{n+1}}{\Xi_{n+1}}}r_{n+2}\right),
\end{equation}
for all $n\ge\ol k.$
Now \eqref{e12} leads to
$$\prod_{k=\ol k}^n\left( r_k+\frac{\b_{k-1}}{1+\frac{\b_{k}}{\Xi_k}}r_{k+1}\right)\ge \prod_{k=\ol k}^n\left(1+\frac{\b_{k+1}}{\Xi_{k+1}}\right)\prod_{k=\ol k}^n\left(r_{k+1}+\frac{\b_k}{1+\frac{\b_{k+1}}{\Xi_{k+1}}}r_{k+2}\right),$$
hence after simplifying we get
\begin{equation}\label{egen}\left(r_{\ol k}+\frac{\b_{\ol k-1}}{1+\frac{\b_{\ol k}}{\Xi_{\ol k}}}r_{\ol k+1}\right)\prod_{k=\ol k}^n\frac{1}{1+\frac{\b_{k+1}}{\Xi_{k+1}}}\ge
r_{n+1}+\frac{\b_n}{1+\frac{\b_{n+1}}{\Xi_{n+1}}}r_{n+2}.
\end{equation}

But, $\frac{\b_{n+1}}{\Xi_{n+1}}=\frac{(n+2)^p}{(n+1)^p}-1,$ hence
$$\prod_{k=\ol k}^n\frac{1}{1+\frac{\b_{k+1}}{\Xi_{k+1}}}=\prod_{k=\ol k}^n\frac{(k+1)^p}{(k+2)^p}=\frac{(\ol k+1)^p}{(n+2)^p}.$$
By denoting $\left(r_{\ol k}+\frac{\b_{\ol k-1}}{1+\frac{\b_{\ol k}}{\Xi_{\ol k}}}r_{\ol k+1}\right)(\ol k+1)^p=a_1$,  we have
$$a_1 \frac{1}{(n+2)^p}\ge r_{n+1}+\frac{\b_n}{1+\frac{1}{n+1}\b_{n+1}}r_{n+2}.$$

Hence,
\begin{equation}\label{e13}
a_1 \frac{1}{n^p}\ge a_1 \frac{1}{(n+1)^p}\ge r_n= g(y_n)-g(\ol x)+\d_n\|x_n-x_{n-1}\|^2\ge g(y_n)-g(\ol x)
\end{equation}
which is (a$_1$).

For (a$_2$) we start from   Lemma \ref{desc} and \eqref{generaldiscrete} and we have
$$g(x_n)-g(y_n)\le\<\n g(y_n),x_n-y_n\>+\frac{L_g}{2}\|x_n-y_n\|^2=$$
$$\frac1s\left\<(x_n-x_{n+1})+\frac{\b n}{n+\a}(x_n-x_{n-1}),-\frac{\b n}{n+\a}(x_n-x_{n-1})\right\>+\frac{L_g}{2}\left(\frac{\b n}{n+\a}\right)^2\|x_n-x_{n-1}\|^2= $$
$$-\left(\frac{\b n}{n+\a}\right)^2\frac{2-sL_g}{2s}\|x_n-x_{n-1}\|^2+\frac1s\left\<x_{n+1}-x_n,\frac{\b n}{n+\a}(x_n-x_{n-1})\right\>.$$
By using the inequality $\<X,Y\>\le \frac12\left(a^2\|X\|^2+\frac{1}{a^2}\|Y\|^2\right)$ for all $X,Y\in\R^m, a\in \R\setminus\{0\}$, we obtain
$$\left\<x_{n+1}-x_n,\frac{\b n}{n+\a}(x_n-x_{n-1})\right\>\le \frac12\left(\frac{1}{2-sL_g}\|x_{n+1}-x_n\|^2+(2-sL_g)\left(\frac{\b n}{n+\a}\right)^2\|x_n-x_{n-1}\|^2\right),$$ consequently
$$g(x_n)-g(y_n)\le\frac{1}{2s(2-sL_g)}\|x_{n+1}-x_n\|^2.$$

From \eqref{ineq} we have
$$\|x_{n}-x_{n-1}\|^2\le \frac{1}{{D}}((g(y_{n})+\d_{n}\|x_{n}-x_{n-1}\|^2)- (g(y_{n+1})+\d_{n+1}\|x_{n+1}-x_{n}\|^2))$$ and since the sequence $(g(y_{n})+\d_{n}\|x_{n+1}-x_{n}\|^2)_{n\ge \ol k}$ is decreasing and has the limit $g(\ol x)$, we obtain that $g(y_{n+1})+\d_{n+1}\|x_{n+1}-x_{n}\|^2\ge g(\ol x)$, consequently
\begin{equation}\label{ee}
\|x_{n}-x_{n-1}\|^2\le\frac{1}{{D}}{r_{n}}.
\end{equation}

Hence, for all $n\ge\ol k$ one has
\begin{equation}\label{ee1}
g(x_n)-g(y_n)\le\frac{1}{2sD(2-sL_g)}r_{n+1}.
\end{equation}

Now, the identity $g(x_n)-g(\ol x)=(g(x_n)-g(y_n))+(g(y_n)-g(\ol x))$ and (a$_1$) lead to
$$g(x_n)-g(\ol x)\le \frac{1}{2sD(2-sL_g)}r_{n+1}+a_1\frac{1}{n^p}$$
for every $n>\ol k$, which combined with \eqref{e13} give
$$g(x_n)-g(\ol x)\le \frac{1}{2sD(2-sL_g)}a_1\frac{1}{(n+2)^p}+a_1\frac{1}{n^p}\le a_1\left(1+\frac{1}{2sD(2-sL_g)}\right)\frac{1}{n^p}=a_2\frac{1}{n^p},$$
for every $n>\ol k$.

For (a$_3$) observe, that  by summing up \eqref{e11} from $n\ge\ol k$ to $P>n$ and using the triangle inequality we obtain
 $$\|x_{P}-x_{n-1}\|\le \sum_{k={n}}^P \|x_{k}-x_{k-1}\|\le$$
  $$-\|x_{n}-x_{n-1}\|+\|x_{P+1}-x_P\|+\frac{9M}{4D}(\varphi(H(y_{n},u_{n})-H(\ol{x},\ol{x}))-\varphi(H(y_{P+1},u_{P+1})-H(\ol{x},\ol{x}))).$$

By  letting $P\To+\infty$ we get
$$\|x_{n-1}-\ol x\|\le -\|x_{n}-x_{n-1}\|+\frac{9M}{4D}\varphi(H(y_{n},u_{n})-H(\ol{x},\ol{x}))\le \frac{9M}{4D}\varphi(H(y_{n},u_{n})-H(\ol{x},\ol{x})).$$
But, $\varphi(t)=\frac{K}{1-\t}t^{1-\t},$ hence
\begin{equation}\label{e1111}
\|x_{n-1}-\ol x\|\le \frac{9MK}{4D(1-\t)}(H(y_{n},u_{n})-H(\ol{x},\ol{x}))^{1-\t}= M_1r_{n}^{1-\t},
\end{equation} where $M_1=\frac{9MK}{4D(1-\t)}.$

But $(*)$ assures that $0\le r_{n}\le 1$ which combined with $\t\in\left(0,\frac12\right]$ leads to
$r_{n}^{1-\t}\le \sqrt{r_{n}}$, consequently we have
$$\|x_{n-1}-\ol x\|\le M_1\sqrt{r_{n}}.$$

The conclusion follow by \eqref{e13}, since  we have
$$\|x_{n}-\ol x\|\le M_1\sqrt{a_1}\frac{1}{n^{\frac{p}{2}}}=a_3\frac{1}{n^{\frac{p}{2}}}$$
for every $n>\ol k.$

Finally, for $n> \ol k$ we have
$$\|y_n-\ol x\|=\left\|x_n+\frac{\b n}{n+\a}(x_n-x_{n-1})-\ol x\right\|\le$$
$$\left(1+\frac{\b  n}{n+\a}\right)\|x_n-\ol x\|+\frac{\b  n}{n+\a}\|x_{n-1}-\ol x\|\le$$
$$\left(1+\frac{\b  n}{n+\a}\right)a_3\frac{1}{n^{\frac{p}{2}}}+\frac{\b  n}{n+\a}a_3\frac{1}{n^{\frac{p}{2}}}\le$$
$$\left(1+2\frac{\b  n}{n+\a}\right)a_3\frac{1}{n^{\frac{p}{2}}}.$$

Let $a_4=(1+2\b)a_3.$ Then
$$\|y_n-\ol x\|\le a_4\frac{1}{n^{\frac{p}{2}}},$$
for all $n> \ol k$, which proves (a$_4$).
\end{proof}

\begin{remark} In the previous theorem we obtained convergence rates with  order $p,$ for every $p>0.$ This happened when we took in \eqref{egen} $$\frac{\b_{n+1}}{\Xi_{n+1}}=\frac{(n+2)^p}{(n+1)^p}-1.$$ But actually we have shown more. If one takes $\frac{\b_{n+1}}{\Xi_{n+1}}=\rho_{n+1}>0$ where $\lim_{n\To+\infty}\rho_n=0$ then one obtains that there exits $\ol k\in\N$  and $A_1>0$ such that for all $n\ge\ol k$ one has
$$\a_n-\b_n+1\ge 1+\rho_{n+1}-\frac{\b_{n-1}}{1+\rho_n}$$ hence \eqref{egen} becomes
$$g(y_n)-g(\ol x)\le A_1\prod_{k=\ol k+1}^{n}\frac{1}{1+\rho_k}.$$
From here, as in the proof of Theorem \ref{th-conv-rate}, one can derive that
$$g(x_n)-g(\ol x)\le A_2\prod_{k=\ol k+1}^{n}\frac{1}{1+\rho_k}\mbox{ for some }A_2>0,$$
and
$$\|x_n-\ol x\|=\mathcal{O}\left(\sqrt{\prod_{k=\ol k+1}^{n}\frac{1}{1+\rho_k}}\right)\mbox{ and }\|y_n-\ol x\|=\mathcal{O}\left(\sqrt{\prod_{k=\ol k+1}^{n}\frac{1}{1+\rho_k}}\right).$$
Having in mind this general result, and taking into account that in \cite{BCL-AA}, for the dynamical system \eqref{eee11} which, as it is shown in Introduction, can be viewed as the continuous counterpart of the  numerical scheme \eqref{generaldiscrete}, it was obtained finite time convergence of the generated trajectories for $\t\in\left(0,\frac12\right)$ and exponential convergence rate for $\t=\frac12$, it seems a valid question whether we can obtain exponential convergence rate for the sequences generated by \eqref{generaldiscrete}, by choosing an appropriate sequence $\rho_n.$ We show in what follow that this is not possible.
We have
$$\prod_{k=\ol k+1}^{n}\frac{1}{1+\rho_k}=e^{-\sum_{k=\ol k+1}^{n}\ln(1+\rho_k)}.$$
Obviously $\ln(1+\rho_k)>0,$ for all $k>\ol k$ and $\lim_{k\To+\infty}\ln(1+\rho_k)=0.$
Now, by using the Ces\`aro-Stolz theorem we obtain that
$$\lim_{n\To+\infty}\frac{\sum_{k=\ol k+1}^{n}\ln(1+\rho_k)}{n}=\lim_{n\To+\infty}\ln(1+\rho_{n+1})=0,$$
hence $\sum_{k=\ol k+1}^{n}\ln(1+\rho_k)=\mathfrak{o}(n),$ which shows that
$$\mathcal{O}\left(\prod_{k=\ol k+1}^{n}\frac{1}{1+\rho_k}\right)>\mathcal{O}\left(e^{-n}\right).$$
\end{remark}

\begin{remark}\label{r3} According to \cite{LP}, $H$ is KL with {\L}ojasiewicz exponent $\t\in\left[\frac12,1\right),$ whenever $g$ is KL with {\L}ojasiewicz exponent $\t\in\left[\frac12,1\right).$ Therefore, we have the following corollary.
\end{remark}

\begin{corollary} In the settings of problem \eqref{opt-pb} consider the sequences $(x_n)_{n\in\N},\,(y_n)_{n\in\N}$ generated by Algorithm \eqref{generaldiscrete}. Assume that $g$ is bounded from below and that $(x_n)_{n\in\N}$ is bounded, let $\ol x\in\crit(g)$ be such that $\lim_{n\To+\infty}x_n=\ol x$ and suppose that $g$
fulfills the {\L}ojasiewicz property at $\ol x$ with {\L}ojasiewicz exponent  $\t=\frac12.$ Then, for every $p>0$ there exist $a_1,a_2,a_3,a_4>0$ and $\ol k\in\N$ such that the following statements hold true:
\begin{itemize}
\item[$(\emph{a}_1)$] $g(y_{n})-g(\ol x)\le a_1 \frac{1}{n^p}$ for  every $n> \ol k$,
\item[$(\emph{a}_2)$] $g(x_{n})-g(\ol x)\le a_2\frac{1}{n^p}$ for  every $n>\ol k$,
\item[$(\emph{a}_3)$] $\|x_n-\ol x\|\le a_3 \frac{1}{n^{\frac{p}{2}}}$ for  every $n>\ol k$,
\item[$(\emph{a}_4)$] $\|y_n-\ol x\|\le a_4 \frac{1}{n^{\frac{p}{2}}}$ for all  $n>\ol k$.
\end{itemize}
\end{corollary}

In case the {\L}ojasiewicz exponent of the regularization function $H$ is $\t\in\left(\frac12,1\right)$ we have the following result concerning the convergence rates of the sequences generated by \eqref{generaldiscrete}.

\begin{theorem}\label{th-conv-rate1}  In the settings of problem \eqref{opt-pb} consider the sequences $(x_n)_{n\in\N},\,(y_n)_{n\in\N}$ generated by Algorithm \eqref{generaldiscrete}. Assume that $g$ is bounded from below and that $(x_n)_{n\in\N}$ is bounded, let $\ol x\in\crit(g)$ be such that $\lim_{n\To+\infty}x_n=\ol x$ and suppose that
$$H:\R^n\times \R^n\To\R,\,H(x,y)=g(x)+\frac{1}{2}\|x-y\|^2$$
fulfills the {\L}ojasiewicz property at $(\ol x,\ol x)\in\crit H$ with {\L}ojasiewicz exponent   $\t\in\left(\frac12,1\right).$ Then, there exist $b_1,b_2,b_3,b_4>0$ such that the following statements hold true:
\begin{itemize}
\item[$(\emph{b}_1)$] $g(y_n)-g(\ol x)\le b_1 \frac{1}{n^{\frac{1}{2\t-1}}},\mbox{ for all }n\ge\ol N_1+2$;
\item[$(\emph{b}_2)$] $g(x_n)-g(\ol x)\le b_2 \frac{1}{n^{\frac{1}{2\t-1}}},\mbox{ for all }n\ge\ol N_1+2$;
\item[$(\emph{b}_3)$] $\|x_{n}-\ol x\|\le b_3 \frac{1}{n^{\frac{1-\t}{2\t-1}}},\mbox{ for all }n\ge\ol N_1+2$;
\item[$(\emph{b}_4)$] $\|y_n-\ol x\|\le  b_4 \frac{1}{n^{\frac{1-\t}{2\t-1}}},\mbox{ for all }n>\ol N_1+2$,
\end{itemize}
where $\ol N_1\in\N$ was defined in the proof of Theorem \ref{th-conv-rate} .
\end{theorem}
\begin{proof}
Also here, to avoid triviality, in what follows we assume  that $H(y_{n},u_{ n})>H(\ol x,\ol x),$ for all $n\ge N.$

  From \eqref{e111} we have that for every $n\ge\ol N_1$ it holds
$$r_n-r_{n+1}\ge \a_n r_{n+1}^{2\t}-\b_n(r_{n+1}-r_{n+2}),$$ where $\a_n= \frac{\Delta_n}{K^2 S_n}\mbox{ and  }\b_n=\frac{2\Delta_n}{s^2S_n\Delta_{n+1}}.$

Hence,
$$(r_n-r_{n+1})r_{n+1}^{-2\t}+\b_n(r_{n+1}-r_{n+2})r_{n+1}^{-2\t}\ge\a_n,$$ for all $n\ge\ol N_1.$

Consider the function $\phi(t)=\frac{K}{2\t-1}t^{1-2\t}$ where $K$ is the constant defined at the {\L}ojasiewicz property of $H$. Then $\phi'(t)=-K t^{-2\t}$ and we have
$$\phi(r_{n+1})-\phi(r_n)=\int_{r_n}^{r_{n+1}}\phi'(t)dt=K\int_{r_{n+1}}^{r_n}t^{-2\t}dt\ge K(r_n-r_{n+1})r_n^{-2\t}.$$
Analogously,
$$\phi(r_{n+2})-\phi(r_{n+1})\ge K(r_{n+1}-r_{n+2})r_{n+1}^{-2\t}.$$

Assume that for some $n\ge\ol N_1$ it holds that $r_n^{-2\t}\ge\frac 12 r_{n+1}^{-2\t}.$

Then
\begin{equation}\label{e14}
\phi(r_{n+1})-\phi(r_n)+\b_n(\phi(r_{n+2})-\phi(r_{n+1}))\ge \frac{K}{2}(r_n-r_{n+1})r_{n+1}^{-2\t}+K\b_n(r_{n+1}-r_{n+2})r_{n+1}^{-2\t}\ge
\end{equation}
$$\frac{K}{2}(r_n-r_{n+1})r_{n+1}^{-2\t}+\frac{K}{2}\b_n(r_{n+1}-r_{n+2})r_{n+1}^{-2\t}\ge\frac{K}{2}\a_n.$$

Conversely, if $2r_n^{-2\t}<r_{n+1}^{-2\t}$ for some $n\ge\ol N_1$, then
$$2^{\frac{2\t-1}{2\t}}r_n^{1-2\t}<r_{n+1}^{1-2\t},$$
hence,
$$\phi(r_{n+1})-\phi(r_n)=\frac{K}{2\t-1}(r_{n+1}^{1-2\t}-r_n^{1-2\t})\ge \frac{K}{2\t-1}\left(2^{\frac{2\t-1}{2\t}}-1\right)r_n^{1-2\t}\ge \frac{K}{2\t-1}\left(2^{\frac{2\t-1}{2\t}}-1\right)r_{\ol N_1}^{1-2\t}=C_1.$$

Consequently, 
$$\phi(r_{n+1})-\phi(r_n)+\b_n(\phi(r_{n+2})-\phi(r_{n+1}))\ge C_1(1+\b_n).$$
Let $C_2=\inf_{n\ge \ol N_1}\frac{C_1(1+\b_n)}{\a_n}>0.$ Then,
\begin{equation}\label{e15}
\phi(r_{n+1})-\phi(r_n)+\b_n(\phi(r_{n+2})-\phi(r_{n+1}))\ge C_2 \a_n.
\end{equation}

From \eqref{e14} and \eqref{e15} we get that there exists $C>0$ such that
$$\phi(r_{n+1})-\phi(r_n)+\b_n(\phi(r_{n+2})-\phi(r_{n+1}))\ge C \a_n,\mbox{ for all }  n\ge\ol N_1.$$
Let $\ol \b=\sup_{n\ge \ol N_1}\b_n.$ Then the  latter relation becomes
$$\phi(r_{n+1})-\phi(r_n)+\ol \b(\phi(r_{n+2})-\phi(r_{n+1}))\ge C \a_n,\mbox{ for all }  n\ge\ol N_1,$$
which leads to
$$\sum_{k=\ol N_1}^n\bigg(\phi(r_{k+1})-\phi(r_k)+\ol \b(\phi(r_{k+2})-\phi(r_{k+1}))\bigg)\ge C \sum_{k=\ol N_1}^n \a_k.$$

Consequently,
$$\phi(r_{n+1})-\phi(r_{\ol N_1})+\ol \b(\phi(r_{n+2})-\phi(r_{\ol N_1+1}))\ge C \sum_{k=\ol N_1}^n \a_k$$
and by using the fact that the sequence $(r_n)_{n\ge\ol N_1}$ is decreasing and $\phi$ is also decreasing, we obtain
$$(1+\ol\b)\phi(r_{n+2})\ge  C \sum_{k=\ol N_1}^n \a_k.$$

In other words
$$r_n^{1-2\t}\ge \frac{C(2\t-1)}{K(1+\ol \b)}\sum_{k=\ol N_1}^{n-2} \a_k,\mbox{ for all }n\ge\ol N_1+2.$$
Hence,
$$r_n\le \left(\frac{C(2\t-1)}{K(1+\ol \b)}\right)^{\frac{-1}{2\t-1}}\left(\sum_{k=\ol N_1}^{n-2} \a_k\right)^{\frac{-1}{2\t-1}},\mbox{ for all }n\ge\ol N_1+2.$$

Since $\sum_{k=\ol N_1}^{n-2} \a_k\ge\ul\a(n-\ol N_1-1)$, where $0<\ul \a=\inf_{k\ge \ol N_1} \a_k$ we have that there exists $M>0$ such that
$$\left(\sum_{k=\ol N_1}^{n-2} \a_k\right)^{\frac{-1}{2\t-1}}\le \ul\a^{\frac{-1}{2\t-1}}(n-\ol N_1-1)^{\frac{-1}{2\t-1}}\le\ul\a^{\frac{-1}{2\t-1}}M n^{\frac{-1}{2\t-1}}, \mbox{ for all }n\ge\ol N_1+2.$$

Therefore, we have
$$r_n\le \left(\frac{C(2\t-1)}{K(1+\ol \b)}\right)^{\frac{-1}{2\t-1}}\ul\a^{\frac{-1}{2\t-1}}M n^{\frac{-1}{2\t-1}}=b_1 n^{\frac{-1}{2\t-1}}, \mbox{ for all }n\ge\ol N_1+2.$$

But, $r_n=g(y_n)-g(\ol x)+\d_n\|x_n-x_{n-1}\|^2$, consequently
$$g(y_n)-g(\ol x)\le b_1 n^{\frac{-1}{2\t-1}},\mbox{ for all }n\ge\ol N_1+2$$ and (b$_1$) is proved.

For (b$_2$) observe that \eqref{ee1} holds  for all $n\ge\ol N_1$, hence for all $n\ge\ol N_1$ one has
$$g(x_n)-g(y_n)\le\frac{1}{2sD(2-sL_g)}r_{n+1}\le \frac{1}{2sD(2-sL_g)}b_1 (n+1)^{\frac{-1}{2\t-1}}.$$

Thus, there exists $M>0$ such that
$$g(x_n)-g(\ol x)=(g(x_n)-g(y_n))+(g(y_n)-g(\ol x))\le  \left(\frac{1}{2sD(2-sL_g)}b_1 M+b_1\right)n^{\frac{-1}{2\t-1}}=b_2n^{\frac{-1}{2\t-1}},$$
 for all $n\ge\ol N_1+2$.

For proving $(b_3)$ we use \eqref{e1111}. Note that the relation
$\|x_{n}-\ol x\|\le M_1r_{n}^{1-\t}$ holds for all $n\ge \ol N_1.$
Hence,
$$\|x_{n}-\ol x\|\le M_1\left(b_1 n^{\frac{-1}{2\t-1}}\right)^{1-\t},\mbox{ for all }n\ge\ol N_1+2.$$
Consequently,
$$\|x_{n}-\ol x\|\le b_3 n^{\frac{\t-1}{2\t-1}},\mbox{ for all }n\ge\ol N_1+2,$$
where $b_3=M_1 b_1^{1-\t}$ and this proves (b$_3$).

For (b$_4$) observe that for $n\ge \ol N_1+3$ we have
$$\|y_n-\ol x\|=\left\|x_n+\frac{\b n}{n+\a}(x_n-x_{n-1})-\ol x\right\|\le$$
$$\left(1+\frac{\b  n}{n+\a}\right)\|x_n-\ol x\|+\frac{\b  n}{n+\a}\|x_{n-1}-\ol x\|\le$$
$$\left(1+\frac{\b  n}{n+\a}\right)b_3 n^{\frac{\t-1}{2\t-1}}+\frac{\b  n}{n+\a} b_3 (n-1)^{\frac{\t-1}{2\t-1}}\le$$
$$\left(\left(1+\frac{\b  n}{n+\a}\right)b_3+\frac{\b  n}{n+\a} b_3\right) (n-1)^{\frac{\t-1}{2\t-1}}\le b_4 n^{\frac{\t-1}{2\t-1}},$$
where one can take $b_4= \sup_{n\ge\ol N_1+3} \left(\left(1+\frac{\b  n}{n+\a}\right)b_3+\frac{\b  n}{n+\a} b_3\right) \left(\frac{n}{n-1}\right)^{\frac{1-\t}{2\t-1}}.$
\end{proof}

According to Remark \ref{r3} we have the following corollary.

\begin{corollary} In the settings of problem \eqref{opt-pb} consider the sequences $(x_n)_{n\in\N},\,(y_n)_{n\in\N}$ generated by Algorithm \eqref{generaldiscrete}. Assume that $g$ is bounded from below and that $(x_n)_{n\in\N}$ is bounded, let $\ol x\in\crit(g)$ be such that $\lim_{n\To+\infty}x_n=\ol x$ and suppose that $g$
fulfills the {\L}ojasiewicz property at $\ol x$ with {\L}ojasiewicz exponent   $\t\in\left(\frac12,1\right).$ Then, there exist $b_1,b_2,b_3,b_4>0$ such that the following statements hold true:
\begin{itemize}
\item[$(\emph{b}_1)$] $g(y_n)-g(\ol x)\le b_1 \frac{1}{n^{\frac{1}{2\t-1}}},\mbox{ for all }n\ge\ol N_1+2$;
\item[$(\emph{b}_2)$] $g(x_n)-g(\ol x)\le b_2 \frac{1}{n^{\frac{1}{2\t-1}}},\mbox{ for all }n\ge\ol N_1+2$;
\item[$(\emph{b}_3)$] $\|x_{n}-\ol x\|\le b_3 \frac{1}{n^{\frac{1-\t}{2\t-1}}},\mbox{ for all }n\ge\ol N_1+2$;
\item[$(\emph{b}_4)$] $\|y_n-\ol x\|\le  b_4 \frac{1}{n^{\frac{1-\t}{2\t-1}}},\mbox{ for all }n>\ol N_1+2$,
\end{itemize}
where $\ol N_1\in\N$ was defined in the proof of Theorem \ref{th-conv-rate}.
\end{corollary}

\section{Conclusions}

  In this paper we show the convergence of a Nesterov type algorithm in a full nonconvex setting, assuming that a regularization of the objective function satisfies the Kurdyka-{\L}ojasiewicz property. For this purpose as a starting point we show a sufficient decrease property for the iterates generated by our algorithm. Though our algorithm is asymptotically equivalent to Nesterov's accelerated gradient method, we cannot obtain full equivalence due to the fact that in order to obtain the above mentioned decrease property we cannot allow the inertial parameter, more precisely the parameter $\b$, to attain the value 1.
Nevertheless, we obtain convergence rates  of order $p$ for every $p>0$, for the sequences generated by our numerical scheme but also for the  function values in these sequences, provided the objective function, or a regularization of the objective function, satisfies the {\L}ojasiewicz property with {\L}ojasiewicz exponent $\t\in\left(0,\frac12\right].$ We also show that, at least with our techniques,  exponential convergence rates cannot be obtained.
In case the  {\L}ojasiewicz exponent of the objective function, or a regularization of the objective function, is $\t\in\left(\frac12,1\right),$ we obtain polynomial convergence rates.

A related future research is the study of a  modified FISTA algorithm in a nonconvex setting. Indeed, let $f:\R^m\To\ol\R$ be a  proper convex and lower semicontinuous function and let $g:\R^m\To\R$ be a (possible nonconvex) smooth function with $L_g$ Lipschitz continuous gradient.
Consider the optimization problem
$$\inf_{x\in\R^m}f(x)+g(x).$$
We associate to this optimization problem the following proximal-gradient algorithm.
For $x_0,y_0\in\R^m$ consider
\begin{equation}\label{fist}\left\{\begin{array}{lll}
\ds x_{n+1}=\prox\nolimits_{s f}(y_n-s\n g(y_n)),\\
\\
\ds y_n=x_n+\frac{\b n}{n+\a}(x_n-x_{n-1}),
\end{array}\right.
\end{equation}
where $\a>0,\,\b\in(0,1)$  and $0<s<\frac{2(1-\b)}{L_g}.$
Obviously, when $f\equiv 0$ then \eqref{fist} becomes the numerical scheme \eqref{generaldiscrete} studied in the present paper.

We emphasize that \eqref{fist} has a  similar formulation as the modified FISTA algorithm studied by Chambolle and Dossal in \cite{ch-do2015} and the convergence of the generated sequences, to a critical point of the objective function $f+g,$  would open the gate for the study of FISTA type algorithms in a nonconvex setting.

\end{document}